\documentclass{amsart}
\usepackage{amsmath}
\usepackage{amsfonts}
\usepackage{amssymb}
\usepackage{latexsym}
\usepackage{graphicx}
\usepackage[final]{hyperref}
\hypersetup{colorlinks=true, linkcolor=blue, anchorcolor=blue, citecolor=red, filecolor=blue, menucolor=blue, pagecolor=blue, urlcolor=blue}

\newcommand{\roundb}[1]{\left(#1\right)} 

\newcommand{\Hproduct}[1]{%
  \roundb{\begin{smallmatrix}#1\end{smallmatrix}}
} %
\newcommand{\LHproduct}[1]{%
  \roundb{\begin{matrix}#1\end{matrix}}
} %

\newcommand{\abs}[1]{\vert #1 \vert}

\newcommand{\norm}[1]{\left\Vert #1 \right\Vert}

\newcommand{\bignorm}[1]{\bigl\Vert #1 \bigr\Vert}

\newcommand{\R}{\mathbb{R}}

\newcommand{\angles}[1]{\langle #1 \rangle}
\newcommand{\bigangles}[1]{\big\langle #1 \big\rangle}

\DeclareMathOperator{\diag}{diag}

\newtheorem{theorem}{Theorem}

\newtheorem{lemma}{Lemma}

\theoremstyle{definition}
\newtheorem{definition}{Definition}

\theoremstyle{remark}
\newtheorem{remark}{Remark}

\numberwithin{equation}{section}
\title[Yang-Mills in Lorenz gauge]{Null structure and local well-posedness in the energy class for the Yang-Mills equations in Lorenz gauge}

\author{Sigmund Selberg}
\address{Department of Mathematical Sciences\\ Norwegian University of Science and Technology\\ N-7491 Trondheim\\ Norway}
\email{sselberg@math.ntnu.no}
\urladdr{www.math.ntnu.no/~sselberg}

\thanks{The first author was supported by the Research Council of Norway, grant no.~213474/F20.}

\author{Achenef Tesfahun}
\address{Department of Mathematical Sciences\\ Norwegian University of Science and Technology\\ N-7491 Trondheim\\ Norway}


\begin{document}

\begin{abstract} 
We demonstrate null structure in the Yang-Mills equations in Lorenz gauge. Such structure was found in Coulomb gauge by Klainerman and Machedon, who used it to prove global well-posedness for finite-energy data. Compared with Coulomb gauge, Lorenz gauge has the advantage---shared with the temporal gauge---that it can be imposed globally in space even for large solutions. Using the null structure and bilinear space-time estimates, we also prove local-in-time well-posedness of the equations in Lorenz gauge, for data with finite energy. The time of existence depends on the initial energy and on the $H^s \times H^{s-1}$-norm of the initial potential, for some $s < 1$.
\end{abstract}

\maketitle

\tableofcontents


\section{Introduction}\label{Intro}

Let $G$ be a compact Lie group and $\mathfrak g$ its Lie algebra. Without loss of generality (see \cite[Theorem 3.28]{Sepanski2007}) we may assume that $G$ is a closed subgroup of $GL(n,\mathbb C)$, hence $\mathfrak g$ has an explicit matrix representation as a subalgebra of $\mathfrak{gl}(n,\mathbb C) = M_{n,n}(\mathbb C)$ (see \cite[Theorem 4.6]{Sepanski2007}), where $M_{n,n}(\mathbb C)$ denotes the set of all $n \times n$ matrices with complex entries.

Given a $\mathfrak g$-valued 1-form $A$ on the Minkowski space-time $\R^{1+3}$, we denote by $F=F[A]$ the associated curvature $F=dA + [A,A]$. That is, given $A_\alpha \colon \R^{1+3} \to \mathfrak g$ for $\alpha \in \{0,1,2,3\}$, define $F_{\alpha\beta} = F[A]_{\alpha\beta} \colon \R^{1+3} \to \mathfrak g$ for $\alpha,\beta \in \{0,1,2,3\}$ by
\begin{equation}\label{Curvature}
  F_{\alpha\beta} = F[A]_{\alpha\beta} = \partial_\alpha A_\beta - \partial_\beta A_\alpha + [A_\alpha,A_\beta],
\end{equation}
where $[\cdot,\cdot]$ denotes the matrix commutator.

In this set-up, the \emph{Yang-Mills equations} read
\begin{equation}\label{YM}
  \partial^\alpha F_{\alpha\beta} + [A^\alpha,F_{\alpha\beta}] = 0 \qquad (\beta \in \{0,1,2,3\}).
\end{equation}
Here the Minkowski metric $\diag(-1,1,1,1)$ on $\R^{1+3}$ is used to raise and lower tensorial indices, thus $A^0=-A_0$ and $A^j = -A_j$ for $j \in \{1,2,3\}$. The coordinates on $\R^{1+3}$ are denoted $(x^0,x^1,x^2,x^3)$, where $t=x^0$ is time and $x=(x^1,x^2,x^3)$ is the spatial variable; $\partial_\alpha$ is the partial derivative with respect to $x^\alpha$, and in particular $\partial_0=\partial_t$ is the time derivative; the spatial gradient $(\partial_1,\partial_2,\partial_3)$ will be denoted $\nabla$, and $\partial = (\partial_t,\nabla)$ is the full space-time gradient. In compliance with Einstein's summation convention, $\alpha$ is implicitly summed over $\{0,1,2,3\}$ in \eqref{YM}; in general, repeated upper/lower greek indices $\alpha,\beta,\gamma,\dots$ (resp.~latin indices $i,j,k,\dots$) are implicitly summed over $\{0,1,2,3\}$ (resp.~$\{1,2,3\}$).

One can think of the Yang-Mills equations as a non-abelian generalization of Maxwell's equations. Indeed, if $G$ is the circle group $U(1)$ (hence $\mathfrak g = \R$), then~\eqref{YM} reduces to Maxwell's equations in vacuum, $\partial^\alpha F_{\alpha\beta} = 0$, for the electromagnetic tensor $F_{\alpha\beta}$. As in the Maxwell theory, it is $F$ that is the interesting quantity, whereas $A$ can be thought of as a potential representing $F$. Moreover, $A$ is not unique, in view of the invariance of the equations under the \emph{gauge transformation}
\begin{equation}\label{GaugeTransformation}
  A_\alpha \longrightarrow A'_\alpha = U A_\alpha U^{-1} - (\partial_\alpha U) U^{-1},
\end{equation}
for a given, sufficiently smooth, $U \colon \R^{1+3} \to G$. Then $F'=F[A']$ is \emph{gauge equivalent} to $F=F[A]$ in the sense that
$$
  F' = U F U^{-1}.
$$
If $(A,F)$ satisfies \eqref{YM}, then so does $(A',F')$, as one can see by reformulating \eqref{YM} in the gauge covariant form,
$$
  D_{(A)}^\alpha F_{\alpha\beta} = 0.
$$
Here $D_{(A)}^\alpha X = \partial^\alpha X + [A^\alpha,X]$ is the gauge covariant derivative, with the property that $X'=UXU^{-1}$ implies $D_{(A')} X' = U [ D_{(A)}X ] U^{-1}$.

\subsection{Gauge conditions}

Identifying gauge equivalent solutions, one has \emph{gauge freedom}, that is, the freedom to choose a representative $(A,F)$ from a given equivalence class. Thus one may complement \eqref{YM} by a condition on $A$, a \emph{gauge condition}. The classical gauge conditions are
\begin{itemize}
  \item temporal: $A_0=0$,
  \item Coulomb: $\partial^i A_i=0$,
  \item Lorenz: $\partial^\alpha A_\alpha=0$,
\end{itemize}
each of which has its advantages and drawbacks.

In both the temporal and Lorenz gauges, the Yang-Mills equations can be expressed as nonlinear wave equations,\footnote{In temporal gauge, $F$ satisfies a nonlinear wave equation where the right-hand side depends on $F(t,x)$, $\int_0^t F(s,x) \, ds$ and their first-order spatial partial derivatives.}, so by classical methods, the Cauchy problem is solvable locally in time for sufficiently regular initial data. For the low-regularity theory, however, it is well-known that null structure plays a crucial role. Klainerman and Machedon \cite{Klainerman1995} discovered null structure in the Coulomb-gauge formulation of the equations. In Coulomb gauge, however, it is not possible to work globally in space unless a smallness condition is imposed, due to the appearance of some nonlinear elliptic equations in the Yang Mills system. For this reason, Klainerman and Machedon worked in temporal gauge, but used local gauge transformations into Coulomb gauge to obtain good estimates for the solution. This leads to considerable technical difficulties.

In summary, there are two properties one would like to have:
\begin{itemize}
\item The Cauchy problem is easy to state and solve globally in space (and locally in time) for data with higher regularity.
\item There is null structure, which allows one to obtain a satisfactory low-regularity theory.
\end{itemize}
Coulomb gauge has the second property, but not the first. Temporal gauge has the first property, and there is also some partial null structure, as proved by Tao \cite{Tao2003}, but the resulting low-regularity theory is limited to small-norm data.

In this paper, we propose to use instead the Lorenz gauge. We demonstrate the presence of null structure in the nonlinear wave equations that appear in this gauge. Thus, Lorenz gauge has both the desirable properties stated above.

Lorenz-gauge null structure was first discovered in \cite{Selberg2010a} for the Maxwell-Dirac equations , and then for the Maxwell-Klein-Gordon equations \cite{Selberg2010b} (see also \cite{Pecher2013}). Subsequently, such structure has been found and applied also in other gauge field theories; see \cite{Bournaveas2012, Selberg2013, Huh2012}.

\subsection{The Cauchy problem} Consider initial data at $t=0$,
$$
  A_\alpha(0) = a_\alpha, \qquad \partial_t A_\alpha(0) = \dot a_\alpha.
$$
Thus,
$$
  F_{\alpha\beta}(0) = f_{\alpha\beta},
$$
where
\begin{equation}\label{f}
\left\{
\begin{aligned}
  f_{ij} &= \partial_i a_j - \partial_j a_i + [a_i,a_j],
  \\
  f_{0i} &= \dot a_i - \partial_i a_0 + [a_0,a_i].
\end{aligned}
\right.
\end{equation}
Note that \eqref{YM} with $\beta=0$ imposes the (gauge invariant) constraint
\begin{equation}\label{DataConstraint}
  \partial^i f_{0i} + [a^i,f_{0i}] = 0.
\end{equation}
In addition, there will be constraints coming from the chosen gauge condition. Let us first look at temporal gauge, then we move on to Lorenz gauge.

Regularity will be measured in the Sobolev spaces $H^s = (I-\Delta)^{-s/2} L^2(\R^3)$.

\begin{definition}\label{TemporalDataDefinition} A temporal-gauge $H^s$ data set consists of any given
$$
  (a_i,\dot a_i) \in H^s \times H^{s-1} \qquad (i \in \{1,2,3\})
$$
and
$$
  a_0=\dot a_0=0,
$$
such that \eqref{DataConstraint} is satisfied:
$$
  \partial^i \dot a_i + [a^i,\dot a_i] = 0.
$$
We say that the data have \emph{finite energy} if $f_{\alpha\beta} \in L^2$, with notation as in \eqref{f}.
\end{definition}

In temporal gauge, Segal \cite{Segal1979} established local well-posedness for $s \ge 3$. Eardley and Moncrief \cite{Eardley1982} improved this to $s \ge 2$ (for the more general Yang-Mills-Higgs equations, in fact), and moreover they were able to prove a global result \cite{Eardley1982b} by using the conservation of energy. Global well-posedness for finite-energy data (that is, $s=1$) was proved by Klainerman and Machedon \cite{Klainerman1995} and extended to Yang-Mills-Higgs by Keel \cite{Keel1997}. Tao \cite{Tao2003} proved local well-posedness for $s > 3/4$, for data with small norm.

As mentioned earlier, although the finite-energy result in \cite{Klainerman1995} is formulated in temporal gauge, Coulomb gauge is used to obtain the main estimates, via a local gauge change based on Uhlenbeck's lemma \cite{Uhlenbeck1982}. In the abelian case, on the other hand, a global Coulomb gauge can be used without problems, as in the works \cite{Klainerman1994, Cuccagna1999, Machedon2004} on the Maxwell-Klein-Gordon system, which is a special case of Yang-Mills-Higgs, corresponding to $G=U(1)$. Local and global regularity properties of the Yang-Mills and Maxwell-Klein-Gordon equations have also been studied in higher space dimensions, and in particular in $1+4$ dimensions, which is the energy-critical case; see \cite{Klainerman1999, Rodnianski2004, Sterbenz2007, Krieger2009}. Structure-preserving numerical schemes for Yang-Mills and Maxwell-Klein-Gordon have been found in \cite{Christiansen2006, Christiansen2011}.

Recently, Oh \cite{Oh2012, Oh2012b} introduced a new approach based on the Yang-Mills heat gauge, and in particular recovered the finite-energy global well-posedness result from \cite{Klainerman1995}.

Let us now turn our attention to the Lorenz gauge.

\begin{definition}\label{LorenzDataDefinition} A Lorenz-gauge $H^s$ data set consists of any given
$$
  (a_\alpha,\dot a_\alpha) \in H^s \times H^{s-1} \qquad (\alpha \in \{0,1,2,3\})
$$
satisfying the Lorenz gauge constraint
$$
  \dot a_0 = \partial^i a_i
$$
and the Yang-Mills constraint \eqref{DataConstraint}. We say that the data have \emph{finite energy} if $f_{\alpha\beta} \in L^2$, with notation as in \eqref{f}.
\end{definition}

\begin{remark}\label{EquivalenceRemark}
In particular, any temporal-gauge data set
$$
  a_0 = \dot a_0 = 0, \qquad (a_i,\dot a_i),
$$
trivially induces a Lorenz-gauge data set $(a',\dot a')$ given by
$$
  a_0' = 0, \qquad \dot a_0' = \partial^i a_i, \qquad (a_i',\dot a_i') = (a_i,\dot a_i).
$$
Moreover, the Lorenz-gauge evolution $(A',F')$ of $(a',\dot a')$ is formally gauge equivalent to the temporal-gauge evolution of $(a,\dot a)$. Indeed, if we can solve
$$
  \partial_t U = - A_0' U, \qquad U(0) = I,
$$
then defining $A$ by the gauge transformation (the inverse of \eqref{GaugeTransformation})
$$
  A_\alpha = U^{-1} A_\alpha' U + U^{-1} \partial_\alpha U,
$$
we have, calculating formally,
$$
  A_0=\partial_t A_0=0,
$$
so $(A,F=F[A])$ is a temporal-gauge solution. Moreover, since $U(0) = I$ and $\partial_t U(0) = -A_0'(0) U(0) = -a_0' = 0$, we have
$$
  A_i(0) = A_i'(0), \qquad \partial_t A_i(0) = \partial_t A_i'(0),
$$
so $(A,F)$ is indeed a temporal-gauge evolution of the original data.
\end{remark}

In Lorenz gauge, $(A,F)$ turns out to satisfy a system of nonlinear wave equations of the form
\begin{equation*}
\left\{
\begin{aligned}
  \square A &= \Pi(A,\partial A) + \Pi(A,F),
  \\
  \square F &= \Pi(A,\partial F) + \Pi(\partial A,\partial A) + \Pi(A,A,F) + \Pi(A,A,A,A),
\end{aligned}
\right.
\end{equation*}
where $\square = \partial^\alpha\partial_\alpha=-\partial_t^2 + \Delta$ is the D'Alembertian and each instance of $\Pi(\dots)$ denotes a multilinear form in the given arguments. The difficult terms here are the bilinear ones. But as we show in the next section, $\Pi(A,\partial A)$, $\Pi(A,\partial F)$ and $\Pi(\partial A,\partial A)$ are all null forms in Lorenz gauge. The term $\Pi(A,F)$, on the other hand, is not. Nevertheless, this term benefits from the regularising effect of the null structure in the wave equation for $F$. This structure implies that $F$ is more regular than $\partial A$, hence $\Pi(A,F)$ is more regular than a generic term $\Pi(A,\partial A)$.

Therefore, although one could expand $F$ using \eqref{Curvature} and reduce to a nonlinear wave equation for $A$ only, of the form
\begin{equation}\label{YMLbad}
  \square A = \Pi(A,\partial A) + \Pi(A,A,A),
\end{equation}
this is not a good idea in Lorenz gauge. In Coulomb gauge, on the other hand, the spatial part $(A_i)$ of $A$ will satisfy an equation of this form where \emph{all} the bilinear terms are null forms, as shown in \cite{Klainerman1995}.

We remark that by using a mixed-norm estimate for the homogeneous wave equation due to Pecher \cite{Pecher1984} (generalising the estimates of Strichartz \cite{Strichartz1977}), one can prove (essentially as in \cite{Ponce1993}) that generic equations of the form \eqref{YMLbad} are locally well posed for $H^s$ data for all $s > 1$, and this is sharp, in view of the counterexamples of Lindblad \cite{Lindblad1996}. If $\Pi(A,\partial A)$ is a null form, however, one can do better, and in particular $s=1$ is allowed, as proved by Klainerman and Machedon \cite{Klainerman1993, Klainerman1994}.

Using the null structure and bilinear space-time estimates, we prove the following local well-posedness result for finite-energy data. Observe that every $H^1$ data set has finite energy.

\begin{theorem}\label{MainThm} Fix $15/16 < s < 1$. Given any Lorenz-gauge $H^1$ data set $(a,\dot a)$ (which necessarily has finite energy), there exist a time $T > 0$, depending on
$$
  \norm{a}_{H^s} + \norm{\dot a}_{H^{s-1}} + \norm{f}_{L^2},
$$
and a solution $(A,F)$ of the Yang-Mills-Lorenz equations on $(-T,T) \times \R^3$ with the given initial data. The solution has the regularity
\begin{align*}
  A &\in C([-T,T]; H^s) \cap C^1([-T,T]; H^{s-1}),
  \\
  F &\in C([-T,T]; L^2),
\end{align*}
and is unique in a certain subspace. The solution depends continuously on the data in the norm $H^s \times H^{s-1} \times L^2$, and higher regularity persists. Moreover, the energy is conserved:
$$
  \mathcal E(t) = \int_{\R^3} \abs{F(t,x)}^2 \, dx = \text{const.}
$$
\end{theorem}

In the next section we prove the key null identities. In section \ref{LGeqs} we write out the wave equations for $A$ and $F$ in Lorenz gauge, and use the null identities to show that all the bilinear terms with derivatives are null forms. The proof of local well-posedness is given in sections \ref{LWP} and \ref{Nonlinear}.

\section{Lorenz-gauge null form identities}\label{Null}

Recall the definition of Klainerman's null forms,
\begin{equation}\label{StandardNullforms}
\left\{
\begin{aligned}
  Q_0(u,v) &= \partial_\alpha u \partial^\alpha v = - \partial_t u \partial_t v + \partial^i u \partial_i v,
  \\
  Q_{\alpha\beta}(u,v) &= \partial_\alpha u \partial_\beta v - \partial_\beta u \partial_\alpha v.
\end{aligned}
\right.
\end{equation}
If $u,v$ are $\mathfrak g$-valued and the product is the matrix product, we can also define commutator versions:\begin{equation}\label{CommutatorNullforms}
\left\{
\begin{aligned}
  Q_0[u,v] &= [\partial_\alpha u, \partial^\alpha v] = Q_0(u,v) - Q_0(v,u),
  \\
  Q_{\alpha\beta}[u,v] &= [\partial_\alpha u, \partial_\beta v] - [\partial_\beta u, \partial_\alpha v] = Q_{\alpha\beta}(u,v) + Q_{\alpha\beta}(v,u).
\end{aligned}
\right.
\end{equation}
Note that $Q_0[\cdot,\cdot]$ is antisymmetric, whereas $Q_{\alpha\beta}[\cdot,\cdot]$ is symmetric. Observe also the identity
\begin{equation}\label{NullformTrick}
  [\partial_\alpha u, \partial_\beta u]
  = \frac12 \left( [\partial_\alpha u, \partial_\beta u] - [\partial_\beta u, \partial_\alpha u] \right)
  = \frac12 Q_{\alpha\beta}[u,u].
\end{equation}

We now prove the key identities that enable us to reveal null structures in Lorenz gauge. The identities involve certain linear combinations of the above null forms. To simplify the notation, we make the following definition.

\begin{definition}\label{NewNull} Let
$$
  \mathfrak Q[u,v] = - \frac12 \varepsilon^{ijk}\varepsilon_{klm} Q_{ij}\left[R^l u^m, v \right]
  - Q_{0i}\left[R^i u_0, v \right],
$$
where $\varepsilon_{ijk}$ is the antisymmetric symbol with $\varepsilon_{123} = 1$ and $R_i = \abs{\nabla}^{-1}\partial_i = (-\Delta)^{-1/2}\partial_i$ are the Riesz transforms.
\end{definition}

\begin{lemma}\label{NullLemma} Assume $A_\alpha,\phi \in \mathcal S$ with values in $\mathfrak g$. Then
\begin{align}
  \label{NullformA}
  A^\alpha \partial_\alpha \phi
  &=
  - \frac12 \varepsilon^{ijk}\varepsilon_{klm} Q_{ij}\left(\abs{\nabla}^{-1} R^l A^m, \phi \right)
  - Q_{0i}\left(\abs{\nabla}^{-1}R^i A_0, \phi\right)
  \\
  \notag
  &\qquad - \abs{\nabla}^{-1} R^i (\partial^\alpha A_\alpha) \partial_i \phi,
  \\
  \label{NullformB}
  (\partial_\alpha \phi) A^\alpha 
  &=
  \frac12 \varepsilon^{ijk}\varepsilon_{klm} Q_{ij}\left(\phi, \abs{\nabla}^{-1} R^l A^m \right)
  + Q_{0i}\left(\phi, \abs{\nabla}^{-1}R^i A_0\right)
  \\
  \notag
  &\qquad - (\partial_i \phi) \abs{\nabla}^{-1} R^i (\partial^\alpha A_\alpha),
  \\
  \label{NullformC}
  [A^\alpha, \partial_\alpha \phi]
  &=
  \mathfrak Q[\abs{\nabla}^{-1} A,\phi] - [ \abs{\nabla}^{-1}R^i (\partial^\alpha A_\alpha), \partial_i \phi],
\end{align}
where $\varepsilon_{ijk}$ is the antisymmetric symbol with $\varepsilon_{123} = 1$, $\abs{\nabla}^{-1} = (-\Delta)^{-1/2}$ and $R_i = \abs{\nabla}^{-1}\partial_i$.
\end{lemma}

\begin{proof}
We only prove \eqref{NullformA}, since an obvious modification of the argument gives \eqref{NullformB}, and then \eqref{NullformC} follows by subtracting the first two.

We split $A$ into its temporal part $A_0$ and its spatial part $\mathbf A = (A_1,A_2,A_3)$, and write the latter as the sum of its curl-free and divergence-free parts $\mathbf A^{\text{cf}}$ and $\mathbf A^{\text{df}}$:
$$
  \mathbf A = - (-\Delta)^{-1}\nabla(\nabla \cdot \mathbf A) + (-\Delta)^{-1} \nabla \times ( \nabla \times \mathbf A) =: \mathbf A^{\text{cf}} + \mathbf A^{\text{df}}.
$$
Now write
$$
  A^\alpha \partial_\alpha \phi
  =
  \left( - A_0 \partial_t \phi
  + \mathbf A^{\text{cf}} \cdot \nabla \phi \right)
  + \mathbf A^{\text{df}} \cdot \nabla \phi
$$
Since $(\mathbf A^{\text{df}})^i = \varepsilon^{ijk} \varepsilon_{klm} R_j  R^l A^m$ one has, as observed in \cite{Klainerman1994},
$$
  \mathbf A^{\text{df}} \cdot \nabla \phi
  =
  \varepsilon^{ijk} \varepsilon_{klm} (R_j  R^l A^m) \partial_i\phi
  =
  - \frac12 \varepsilon^{ijk}\varepsilon_{klm} Q_{ij}\left(\abs{\nabla}^{-1} R^l A_m, \phi\right).
$$
Next, writing
$$
  \mathbf A^{\text{cf}} = - (-\Delta)^{-1}\nabla(\nabla \cdot \mathbf A) = - (-\Delta)^{-1}\nabla(\partial_t A_0 + \partial^\alpha A_\alpha),
$$
we see that
\begin{align*}
  &- A_0 \partial_t \phi
  + \mathbf A^{\text{cf}} \cdot \nabla \phi
  =
  - A_0 \partial_t \phi
  - (-\Delta)^{-1}\partial^i(\partial_t A_0 + \partial^\alpha A_\alpha) \partial_i \phi
  \\
  &\qquad= - A_0 \partial_t \phi
  - (-\Delta)^{-1}\partial^i(\partial_t A_0) \partial_i \phi
  - (-\Delta)^{-1}\partial^i(\partial^\alpha A_\alpha) \partial_i \phi
  \\
  &\qquad=  (-\Delta)^{-1} (\partial^i \partial_i A_0) \partial_t \phi
  - (-\Delta)^{-1}\partial^i(\partial_t A_0) \partial_i \phi
  - \abs{\nabla}^{-1} R^i (\partial^\alpha A_\alpha) \partial_i \phi
  \\
  &\qquad=  \partial_i \left( \abs{\nabla}^{-1} R^i  A_0 \right) \partial_t \phi
  - \partial_t \left( \abs{\nabla}^{-1}R^i A_0 \right) \partial_i \phi
  - \abs{\nabla}^{-1} R^i (\partial^\alpha A_\alpha) \partial_i \phi
  \\
  &\qquad=
  - Q_{0i}\left(\abs{\nabla}^{-1}R^i A_0, \phi\right)
  - \abs{\nabla}^{-1} R^i (\partial^\alpha A_\alpha) \partial_i \phi,
\end{align*}
and this completes the proof of \eqref{NullformA}.
\end{proof}

We also need the identity
\begin{equation}\label{KeyIdentity2}
  [\partial_t A^\alpha, \partial_\alpha \phi]
  =
  [\partial^\alpha A_\alpha, \partial_t \phi] + Q_{0i}[A^i,\phi].
\end{equation}
Indeed, we calculate
\begin{align*}
  [\partial_t A^\alpha, \partial_\alpha \phi]
  &=
  [-\partial_t A_0, \partial_t \phi]
  + [\partial_t A^i, \partial_i \phi]
  \\
  &=
  [\partial^\alpha A_\alpha - \partial^i A_i, \partial_t \phi]
  + [\partial_t A^i, \partial_i \phi]
  \\
  &=
  [\partial^\alpha A_\alpha, \partial_t \phi] - [\partial_i A^i, \partial_t \phi]
  + [\partial_t A^i, \partial_i \phi]
  \\
  &=
  [\partial^\alpha A_\alpha, \partial_t \phi] + Q_{0i}[A^i,\phi].
\end{align*}

\section{The equations in Lorenz gauge}\label{LGeqs}

\subsection{Wave equation for $A$}

In view of \eqref{Curvature}, the Yang-Mills equations \eqref{YM} can be written as
$$
  \square A_\beta - \partial_\beta(\partial^\alpha A_\alpha) + [\partial^\alpha A_\alpha,A_\beta] + [A_\alpha,\partial^\alpha A_\beta] + [A^\alpha,F_{\alpha\beta}]
  = 0,
$$
where $\square = \partial^\alpha\partial_\alpha=-\partial_t^2 + \Delta$. Imposing the Lorenz gauge condition
$$
  \partial^\alpha A_\alpha = 0,
$$
this simplifies to
\begin{equation}\label{YML}
  \square A_\beta = - [A^\alpha,\partial_\alpha A_\beta] - [A^\alpha,F_{\alpha\beta}],
\end{equation}
but in view of Lemma \ref{NullLemma}, the first term on the right-hand side is a null form:
\begin{equation}\label{YMLnull}
  \square A_\beta = - [\mathfrak Q[\abs{\nabla}^{-1} A,A_\beta] - [A^\alpha,F_{\alpha\beta}].
\end{equation}

Expanding the last term in \eqref{YML}, one could also write
$$
  \square A_\beta = - 2[A^\alpha,\partial_\alpha A_\beta] + [A^\alpha,\partial_\beta A_\alpha] - [A^\alpha, [A_\alpha,A_\beta]],
$$
but this is not a good idea. The cubic term causes no problems, but the new bilinear term $[A^\alpha,\partial_\beta A_\alpha]$ is not a null form, as far as we know,\footnote{This is in contrast to the situation in Coulomb gauge, since there one can apply the projection $\mathbf P$ onto divergence-free fields on both sides of the wave equation for $A_j$, and use the fact that $\mathbf P[A^i,\partial_j A_i]$ is a null form; see \cite{Klainerman1994, Klainerman1995}.} and in fact it is worse than the term $[A^\alpha,F_{\alpha\beta}]$, since $F$ has better regularity than $\partial A$. The reason for this is that $F$ itself satisfies a nonlinear wave equation with null structure in the bilinear terms, as will soon transpire.

\subsection{Wave equation for $F$}

Regardless of the choice of gauge, $F$ satisfies
\begin{equation}\label{CurvatureEq}
  \square F_{\beta\gamma} + [A^\alpha,\partial_\alpha F_{\beta\gamma}] + \partial^\alpha[A_\alpha,F_{\beta\gamma}] + \left[A^\alpha,[A_\alpha,F_{\beta\gamma}]\right] + 2[F^{\alpha}_{{\;\;\,}\beta},F_{\gamma\alpha}] = 0.
\end{equation}
We recall the derivation below. The initial conditions are
$$
  (F,\partial_t F)(0) = (f,\dot f),
$$
where
\begin{equation}\label{Fdata}
\left\{
\begin{aligned}
  f_{ij} &= \partial_i a_j - \partial_j a_i + [a_i,a_j],
  \\
  f_{0i} &= \dot a_i - \partial_i a_0 + [a_0,a_i],
  \\
  \dot f_{ij} &= \partial_i \dot a_j - \partial_j \dot a_i + [\dot a_i,a_j] + [a_i,\dot a_j],
  \\
  \dot f_{0i} &= \partial^j f_{ji} + [a^\alpha,f_{\alpha i}],
\end{aligned}
\right.
\end{equation}
the expression for $\dot f_{0i}$ coming from \eqref{YM} with $\beta=i$.

Expanding the last term in \eqref{CurvatureEq} yields
\begin{equation*}
\left\{
\begin{aligned}
  \square F_{\beta\gamma} = &- [\partial^\alpha A_\alpha,F_{\beta\gamma}] 
  \\
  &- 2[A^\alpha,\partial_\alpha F_{\beta\gamma}]
  + 2[\partial_\gamma A^\alpha, \partial_\alpha A_\beta]
  - 2[\partial_\beta A^\alpha, \partial_\alpha A_\gamma]
  \\
  &+ 2[\partial^\alpha A_\beta , \partial_\alpha A_\gamma]
  + 2[\partial_\beta A^\alpha, \partial_\gamma A_\alpha]
  \\
  &- [A^\alpha,[A_\alpha,F_{\beta\gamma}]] + 2[F_{\alpha\beta},[A^\alpha,A_\gamma]] - 2[F_{\alpha\gamma},[A^\alpha,A_\beta]]
  \\
  &- 2[[A^\alpha,A_\beta],[A_\alpha,A_\gamma]].
\end{aligned}
\right.
\end{equation*}
Imposing the Lorenz gauge condition, the first term on the right-hand side disappears, the second term is a null form by Lemma \ref{NullLemma}, the third and fourth terms are null forms by either Lemma \ref{NullLemma} or the identity \eqref{KeyIdentity2}, the fifth term is identical to $2Q_0[A_\beta,A_\gamma]$, and the sixth term equals $Q_{\beta\gamma}[A^\alpha,A_\alpha]$ by the identity \eqref{NullformTrick}.

The conclusion is that, in Lorenz gauge,
\begin{equation}\label{CurvatureEqLorenz2}
\left\{
\begin{aligned}
  \square F_{ij} = &- 2\mathfrak Q[\abs{\nabla}^{-1} A,F_{ij}]
  + 2\mathfrak Q[\abs{\nabla}^{-1} \partial_j A, A_i]
  - 2\mathfrak Q[\abs{\nabla}^{-1} \partial_i A, A_j]
  \\
  &+ 2Q_0[A_i , A_j]
  + Q_{ij}[A^\alpha,A_\alpha]
  \\
  & - [A^\alpha,[A_\alpha,F_{ij}]] + 2[F_{\alpha i},[A^\alpha,A_j]] - 2[F_{\alpha j},[A^\alpha,A_i]]
  \\
  & - 2[[A^\alpha,A_i],[A_\alpha,A_j]].
\end{aligned}
\right.
\end{equation}
and
\begin{equation}\label{CurvatureEqLorenz3}
\left\{
\begin{aligned}
  \square F_{0i} = &- 2\mathfrak Q[\abs{\nabla}^{-1} A,F_{0i}]
  + 2\mathfrak Q[\abs{\nabla}^{-1} \partial_i A, A_0]
  - 2 Q_{0j}[A^j,A_i]
  \\
  &+ 2Q_0[A_0 , A_i]
  + Q_{0i}[A^\alpha,A_\alpha]
  \\
  & - [A^\alpha,[A_\alpha,F_{0i}]] + 2[F_{\alpha 0},[A^\alpha,A_i]] - 2[F_{\alpha i},[A^\alpha,A_0]]
  \\
  & - 2[[A^\alpha,A_0],[A_\alpha,A_i]].
\end{aligned}
\right.
\end{equation}
This completes the derivation of the null structure.

To end this section, we recall the derivation of \eqref{CurvatureEq}. Let $D_\alpha$ be the covariant derivative $D_\alpha = \partial_\alpha + [A_\alpha,\cdot]$, and note the commutation identity
$$
  D_\alpha D_\beta X - D_\beta D_\alpha X = [F_{\alpha\beta},X],
$$
which follows from the Jacobi identity
$$
  [X,[Y,Z]] + [Z,[X,Y]] + [Y,[Z,X]] = 0.
$$
In particular, $D^\alpha D_\beta F_{\gamma\alpha} = D_\beta D^\alpha F_{\gamma\alpha} + [F^{\alpha}_{{\;\;\,}\beta},F_{\gamma\alpha}]$ and $D^\alpha D_\gamma F_{\alpha\beta} = D_\gamma D^\alpha F_{\alpha\beta} + [F^{\alpha}_{{\;\;\,}\gamma},F_{\alpha\beta}]$. Therefore, applying $D^\alpha$ to both sides of the Bianchi identity
$$
  D_\alpha F_{\beta\gamma} + D_\beta F_{\gamma\alpha} + D_\gamma F_{\alpha\beta} = 0
$$
yields
$$
  D^\alpha D_\alpha F_{\beta\gamma} + D_\beta D^\alpha F_{\gamma\alpha} + [F^{\alpha}_{{\;\;\,}\beta},F_{\gamma\alpha}] + D_\gamma D^\alpha F_{\alpha\beta} + [F^{\alpha}_{{\;\;\,}\gamma},F_{\alpha\beta}] = 0.
$$
But by the Yang-Mills equation, $D^\alpha F_{\gamma\alpha} = 0$ and $D^\alpha F_{\alpha\beta} = 0$. By antisymmetry, $[F^{\alpha}_{{\;\;\,}\gamma},F_{\alpha\beta}] = [F_{\alpha\gamma},F^{\alpha}_{{\;\;\,}\beta}] = - [F^{\alpha}_{{\;\;\,}\beta}, F_{\alpha\gamma}] = [F^{\alpha}_{{\;\;\,}\beta}, F_{\gamma\alpha}]$. Thus,
$$
  D^\alpha D_\alpha F_{\beta\gamma} + 2[F^{\alpha}_{{\;\;\,}\beta},F_{\gamma\alpha}] = 0.
$$
Finally, in view of the identity
$$
  D^\alpha D_\alpha X = \square X + [A^\alpha,\partial_\alpha X] + \partial^\alpha[A_\alpha,X] + [A^\alpha,[A_\alpha,X]],
$$
one obtains \eqref{CurvatureEq}.

\section{Local well-posedness}\label{LWP}

Here we prove Theorem \ref{MainThm}.

By iteration, we solve simultaneously the nonlinear wave equations for $A$ and $F$, written in terms of null forms as in \eqref{YMLnull}, \eqref{CurvatureEqLorenz2} and \eqref{CurvatureEqLorenz3}:
\begin{equation}\label{AFwaves}
\left\{
\begin{aligned}
  \square A_\beta &= \mathfrak M_\beta(A,\partial_t A,F,\partial_t F),
  \\
  \square F_{\beta\gamma} &= \mathfrak N_{\beta\gamma}(A,\partial_t A,F,\partial_t F),
\end{aligned}
\right.
\end{equation}
where
\begin{gather*}
  \mathfrak M_\beta(A,\partial_t A,F,\partial_t F) = - \mathfrak Q[\abs{\nabla}^{-1} A,A_\beta] - [A^\alpha,F_{\alpha\beta}],
  \\
  \left\{
  \begin{aligned}
  \mathfrak N_{ij}(A,\partial_t A,F,\partial_t F)
  = &- 2\mathfrak Q[\abs{\nabla}^{-1} A,F_{ij}]
  + 2\mathfrak Q[\abs{\nabla}^{-1} \partial_j A, A_i]
  \\
  &- 2\mathfrak Q[\abs{\nabla}^{-1} \partial_i A, A_j] + 2Q_0[A_i , A_j]
  + Q_{ij}[A^\alpha,A_\alpha]
  \\
  & - [A^\alpha,[A_\alpha,F_{ij}]] + 2[F_{\alpha i},[A^\alpha,A_j]] - 2[F_{\alpha j},[A^\alpha,A_i]]
  \\
  & - 2[[A^\alpha,A_i],[A_\alpha,A_j]],
  \end{aligned}
  \right.
  \\
  \left\{
  \begin{aligned}
  \mathfrak N_{0i}(A,\partial_t A,F,\partial_t F)
  = &- 2\mathfrak Q[\abs{\nabla}^{-1} A,F_{0i}]
  + 2\mathfrak Q[\abs{\nabla}^{-1} \partial_i A, A_0]
  \\
  &- 2 Q_{0j}[A^j,A_i] + 2Q_0[A_0 , A_i]
  + Q_{0i}[A^\alpha,A_\alpha]
  \\
  & - [A^\alpha,[A_\alpha,F_{0i}]] + 2[F_{\alpha 0},[A^\alpha,A_i]] - 2[F_{\alpha i},[A^\alpha,A_0]]
  \\
  & - 2[[A^\alpha,A_0],[A_\alpha,A_i]].
  \end{aligned}
  \right.
\end{gather*}
The initial conditions are
\begin{equation}\label{AFdata}
  (A,\partial_t A)(0) = (a,\dot a),
  \qquad
  (F,\partial_t F)(0) = (f,\dot f).
\end{equation}

We shall prove the following.

\begin{theorem}\label{LWPThm} (Local well-posedness.)
Let $15/16 < s < 1$. Given any data
$$
  (a,\dot a) \in H^s \times H^{s-1},
  \qquad
  (f,\dot f) \in L^2 \times H^{-1},
$$
there exists a $T > 0$, depending continuously on the data norm
$$
  \norm{a}_{H^s} + \norm{\dot a}_{H^{s-1}} + \norm{f}_{L^2} + \bignorm{\dot f}_{H^{-1}},
$$
and there exists
\begin{equation}\label{SolutionSpace}
\left\{
\begin{aligned}
  A &\in C([-T,T]; H^s) \cap C^1([-T,T]; H^{s-1}),
  \\
  F &\in C([-T,T]; L^2) \cap C^1([-T,T]; H^{-1}),
\end{aligned}
\right.
\end{equation}
solving \eqref{AFwaves} on $S_T = (-T,T) \times \R^3$ in the sense of distributions, and satisfying the initial condition \eqref{AFdata}.

The solution has the regularity, for some $b > 1/2$,
$$
  \left( A \pm \frac{1}{i\angles{\nabla}} \partial_t A \right) \in X_\pm^{s,b}(S_T),
  \qquad
  \left( F \pm \frac{1}{i\angles{\nabla}} \partial_t F \right) \in X_\pm^{0,b}(S_T),
$$
and it is the unique solution with this property. (See Definition \ref{Xdef} below for the definition of the spaces used here.)

The solution depends continuously on the data. Moreover, higher regularity persists, in the sense that if, for some $k \in \mathbb N$, we have that
$$
  \nabla^\alpha(a,\dot a) \in H^s \times H^{s-1}, \qquad \nabla^{\alpha}(f,\dot f) \in L^2 \times H^{-1},
$$
for all multi-indexes $\alpha \in \mathbb N_0^3$ with $\abs{\alpha} \le k$, then $\partial^\alpha(A,F)$ belongs to \eqref{SolutionSpace} for all $\alpha \in \mathbb N_0^{1+3}$ with $\abs{\alpha} \le k$. In particular, if the data are smooth and compactly supported, then the solution is smooth.
\end{theorem}

\begin{remark}
In this theorem, we do not assume any compatibility conditions on the data, hence $F=F[A]$ and the Lorenz gauge condition $\partial^\alpha A_\alpha=0$ will not necessarily hold. They will hold, however, if we assume the constraints \eqref{Fdata}, \eqref{DataConstraint} and $\dot a_0 = \partial^i a_i$. Indeed, setting
$$
  u = \partial^\alpha A_\alpha, \qquad V = F-F[A],
$$
then
$$
  u(0) = \partial_t u(0) = 0,
  \qquad
  V(0) = \partial_t V(0) =0,
$$
and $(u,V)$ satisfies a system of wave equations of the form
\begin{equation}\label{uVwaves}
\left\{
\begin{aligned}
  \square u &= \Pi(A,\partial u) + \Pi(\partial A,V) + \Pi(A,\partial V)
  \\
  &\quad + \Pi(\partial \nabla A,\abs{\nabla}^{-1}Ru) + \Pi(\nabla A,\partial \abs{\nabla}^{-1}Ru)
  \\
  &\quad  + \Pi(A,A,u) + \Pi(A,A,V)  + \Pi(A,\partial A,\abs{\nabla}^{-1}Ru),
  \\
  \square V &= \Pi(\partial A,V) + \Pi(A,\partial V) + \Pi(\partial G, \abs{\nabla}^{-1}Ru)
  \\
  &\quad + \Pi(\partial \nabla A,\abs{\nabla}^{-1}Ru) + \Pi(\nabla A,\partial \abs{\nabla}^{-1}Ru)
  \\
  &\quad + \Pi(A,A,V) + \Pi(A,\partial A,\abs{\nabla}^{-1}Ru),
\end{aligned}
\right.
\end{equation}
where $R=(R_1,R_2,R_3)$, $R_i=\abs{\nabla}^{-1}\partial_i$, and each instance of $\Pi(\dots)$ denotes a multilinear form in the given arguments. By regularisation of the initial data (as in \cite[Proposition 1.2]{Klainerman1995}), persistence of higher regularity, and continuous dependence on the data, we may assume smoothness. Therefore, the unique solution is $(u,V) = (0,0)$. Thus $(A,F=F[A])$ is a solution of the actual Yang-Mills equations \eqref{YM}, and the Lorenz gauge condition holds. The regularisation argument shows that our solutions are limits of smooth solutions, and since conservation of energy holds for the latter, it also holds for our solutions.
\end{remark}

Let us now turn to the proof of Theorem \ref{LWPThm}.

It is convenient to recast the system in first-order form, since this will allow us to treat the null forms in a unified way (see section \ref{NullformReductions}). To avoid certain singularities at low frequency, we first rewrite our system so that we have the Klein-Gordon operator $\square-1$ on the left-hand side:
$$
\left\{
\begin{aligned}
  (\square-1) A_\beta &= -A_\beta + \mathfrak M_\beta(A,\partial_t A,F,\partial_t F),
  \\
  (\square-1) F_{\beta\gamma} &= -F_{\beta\gamma} + \mathfrak N_{\beta\gamma}(A,\partial_t A,F,\partial_t F).
\end{aligned}
\right.
$$
Now apply the change of variables $(A,\partial_t A,F,\partial_t F) \to (A_+,A_-,F_+,F_-)$ given by
$$
  A_\pm = \frac12 \left( A \pm \frac{1}{i\angles{\nabla}} \partial_t A \right),
  \qquad
  F_\pm = \frac12 \left( F \pm \frac{1}{i\angles{\nabla}} \partial_t F \right).
$$
Equivalently,
\begin{equation}\label{Substitution}
  (A,\partial_t A,F,\partial_t F)
  =
  \bigl( A_+ + A_-, i\angles{\nabla} (A_+ - A_-),
  F_+ + F_-, i\angles{\nabla} (F_+ - F_-) \bigr).
\end{equation}
Then our system transforms to
\begin{equation}\label{AFwavesSplit}
\left\{
\begin{aligned}
  (i\partial_t + \angles{\nabla}) A_+ &= - \frac{1}{2 \angles{\nabla}} \mathfrak M'(A_+,A_-,F_+,F_-),
  \\
  (i\partial_t - \angles{\nabla}) A_- &= + \frac{1}{2 \angles{\nabla}} \mathfrak M'(A_+,A_-,F_+,F_-),
  \\
  (i\partial_t + \angles{\nabla}) F_+ &= - \frac{1}{2 \angles{\nabla}} \mathfrak N'(A_+,A_-,F_+,F_-),
  \\
  (i\partial_t - \angles{\nabla}) F_- &= + \frac{1}{2 \angles{\nabla}} \mathfrak N'(A_+,A_-,F_+,F_-),
\end{aligned}
\right.
\end{equation}
where
\begin{align*}
  \mathfrak M'(A_+,A_-,F_+,F_-) &= - (A_+ + A_-) + \mathfrak M(A,\partial_t A,F,\partial_t F),
  \\
  \mathfrak N'(A_+,A_-,F_+,F_-) &= - (F_+ + F_-) + \mathfrak N(A,\partial_t A,F,\partial_t F),
\end{align*}
and in the right-hand side it is understood that we use the substitution \eqref{Substitution} on the arguments of $\mathfrak M$ and $\mathfrak N$.

The transformed initial data are
\begin{equation}\label{AFdataSplit}
\left\{
\begin{aligned}
  A_\pm(0) &= a_\pm := \frac12 \left( a \pm \frac{1}{i\angles{\nabla}} \dot a \right) \in H^s,
  \\
  F_\pm(0) &= f_\pm := \frac12 \left( f \pm \frac{1}{i\angles{\nabla}} \dot f \right) \in L^2.
\end{aligned}
\right.
\end{equation}

We solve the transformed system by iterating in $X^{s,b}$-spaces associated to the dispersive operators $i\partial_t \pm \angles{\nabla}$. Spaces of this type have become an indispensable tool in the study of nonlinear dispersive PDEs since the seminal work of Bourgain \cite{Bourgain1993}. For an exposition of the theory, see \cite[Section 2.6]{Tao2006}.

\begin{definition}\label{Xdef} For $(s,b) \in \R^2$, define $X^{s,b}_\pm$ to be the completion of $\mathcal S(\R^{1+3})$ with respect to the norm
$$
  \norm{u}_{X^{s,b}_\pm} = \norm{\angles{\xi}^s \bigangles{-\tau \pm \angles{\xi}}^b \widetilde u(\tau,\xi)}_{L^2_{\tau,\xi}},
$$
where $\widetilde u(\tau,\xi) = \mathcal F_{t,x} u(\tau,\xi)$ is the space-time Fourier transform of $u(t,x)$. The restriction to a time slab $S_T = (-T,T) \times \R^3$, denoted $X^{s,b}_\pm(S_T)$, is defined as the quotient space $X^{s,b}_\pm / \mathcal M$, where $\mathcal M$ is the closed subspace consisting of $u \in X^{s,b}_\pm$ such that $u = 0$ on $S_T$.
\end{definition}

We recall some facts about $X^{s,b}$-spaces (see, e.g., \cite[Section 2.6]{Tao2006}).

\begin{lemma}\label{EmbeddingLemma} Let $s \in \R$, $b > 1/2$ and $T > 0$. Then $X^{s,b}_\pm(S_T) \hookrightarrow C([-T,T];H^s)$.
\end{lemma}

\begin{lemma}\label{IVPlemma}
Let $s \in \R$, $1/2 < b \le b' < 1$ and $0 < T < 1$. Then for any $f \in H^s$ and $F \in X^{s,b'-1}_\pm(S_T)$, the Cauchy problem
$$
  (i\partial_t \pm \angles{\nabla}) u = G \quad \text{on $S_T$}, \qquad u(0) = g,
$$
has a unique solution
$$
  u \in X^{s,b}_\pm(S_T).
$$
Moreover,
$$
  \norm{u}_{X^{s,b}_\pm(S_T)} \le C \left( \norm{g}_{H^s} + T^{b'-b} \norm{G}_{X^{s,b'-1}_\pm(S_T)} \right),
$$
where $C$ depends only on $b$ and $b'$.
\end{lemma}

\begin{remark} For sufficiently regular $G$ (say $G \in C([-T,T];H^s)$), the solution in the last lemma is
$u(t) = e^{\pm i t \angles{\nabla}} g - i \int_0^t e^{\pm i (t-t') \angles{\nabla}} G(t') \, dt'$ for $-T \le t \le T$.
\end{remark}

Using Lemma \ref{IVPlemma} and a standard iteration argument (which we omit), local well-posedness can be deduced from the following nonlinear estimates.

\begin{lemma}\label{NonlinearEstimates}
Let $0 < T < 1$, $b = 1/2 + \varepsilon$, $b'=1/2+2\varepsilon$ and $1-\varepsilon < s < 1$, where $0 < \varepsilon \le 1/16$. Then we have the estimates
\begin{align*}
  \norm{\mathfrak M'\left(A_+,A_-,F_+,F_-\right)}_{X^{s-1,b'-1}_\pm(S_T)}
  &\lesssim N (1+N^3),
  \\
  \norm{\mathfrak N'\left(A_+,A_-,F_+,F_-\right)}_{X^{-1,b'-1}_\pm(S_T)}
  &\lesssim N (1+N^3),
\end{align*}
where
$$
  N = \norm{A_+}_{X^{s,b}_+(S_T)} + \norm{A_-}_{X^{s,b}_-(S_T)}
  + \norm{F_+}_{X^{0,b}_+(S_T)} + \norm{F_-}_{X^{0,b}_-(S_T)}.
$$
Moreover, we have the difference estimates
\begin{align*}
  \norm{\mathfrak M'\left(A_+,A_-,F_+,F_-\right) - \mathfrak M'\left(A_+',A_-',F_+',F_-'\right)}_{X^{s-1,b'-1}_\pm(S_T)}
  &\lesssim \delta (1+N^3),
  \\
  \norm{\mathfrak M'\left(A_+,A_-,F_+,F_-\right) - \mathfrak M'\left(A_+',A_-',F_+',F_-'\right)}_{X^{-1,b'-1}_\pm(S_T)}
  &\lesssim \delta (1+N^3),
\end{align*}
where
$$
  \delta
  =
  \sum_{\pm} \left(
  \norm{A_\pm - A_\pm'}_{X^{s,b}_\pm(S_T)}
  + \norm{F_\pm - F_\pm'}_{X^{0,b}_\pm(S_T)} \right).
$$
\end{lemma}

Thus, by iteration we obtain a solution $(A_+,A_-,F_+,F_-)$ of the transformed system on $S_T$ for $T > 0$ sufficiently small (the relevant condition is $T^\varepsilon (1+N^3) \ll 1$). The solution has the regularity
$$
  A_\pm \in X^{s,b}_\pm(S_T), \qquad F_\pm \in X^{0,b}_\pm(S_T),
$$
and is unique in this space. By Lemma \ref{EmbeddingLemma},
\begin{equation}\label{SolutionSpaceSplit}
  A_\pm \in C([-T,T]; H^s),
  \qquad
  F_\pm \in C([-T,T]; L^2),
\end{equation}
Standard arguments also give continuous dependence on the data and persistence of higher regularity. We omit the details.

Finally, we can transform back to the original formulation of the system by defining $A = A_+ + A_-$ and $F = F_+ + F_-$. Pairwise addition of the equations in \eqref{AFwavesSplit} reveals that $\partial_t A = i\angles{\nabla} (A_+ - A_-)$ and $\partial_t F = i\angles{\nabla} (F_+ - F_-)$. Thus, $\mathfrak M'(A_+,A_-,F_+,F_-)  = - A + \mathfrak M(A,\partial_t A, F, \partial_t F)$ and similarly for $\mathfrak N'$ and $\mathfrak N$. Since $\square-1$ can be factored as either $(i\partial_t + \angles{\nabla})(i\partial_t - \angles{\nabla})$ or $(i\partial_t - \angles{\nabla}) (i\partial_t + \angles{\nabla})$, we now see that \eqref{AFwavesSplit} implies \eqref{AFwaves}. Moreover, the map $(A_+,A_-,F_+,F_-) \to (A,\partial_t A,F,\partial_t F)$ is continuous from the space \eqref{SolutionSpaceSplit} to the space \eqref{SolutionSpace}.

It remains to prove the estimates in Lemma \ref{NonlinearEstimates}. Note that $\mathfrak M'$ and $\mathfrak N'$ contain linear terms for which the estimates are trivial. We now turn our attention to the estimates for the nonlinear parts, $\mathfrak M$ and $\mathfrak N$.

\section{Estimates for the nonlinear terms}\label{Nonlinear}

Here we prove Lemma \ref{NonlinearEstimates}. In addition to the $X^{s,b}_\pm$-norms defined in the last section, we will need the wave-Sobolev norms
$$
  \norm{u}_{H^{s,b}} = \norm{\angles{\xi}^s \angles{\abs{\tau} - \abs{\xi}}^b \widetilde u(\tau,\xi)}_{L^2_{\tau,\xi}}.
$$
Norms of this this type were first applied to the study of regularity questions for nonlinear wave equations with null forms by Klainerman and Machedon \cite{Klainerman1993}.

Note the relations
\begin{equation}\label{HXH}
\left\{
\begin{alignedat}{2}
  \norm{u}_{H^{s,b}} &\le \norm{u}_{X^{s,b}_\pm}& \quad &\text{if $b \ge 0$},
  \\
  \norm{u}_{X^{s,b}_\pm} &\le \norm{u}_{H^{s,b}}& \quad &\text{if $b \le 0$},
\end{alignedat}
\right.
\end{equation}
between the two types of norms.

Since the operators $\mathfrak M$ and $\mathfrak N$ are local in time, it suffices to prove the estimates in Lemma \ref{NonlinearEstimates} without the restriction to $S_T$. Since $b'-1 < 0$, we may replace the $X$-norms on the left-hand sides of the estimates by the corresponding $H$-norms, in view of \eqref{HXH}. We will only prove the first two estimates in the lemma; the difference estimates follow from the same arguments, in view of the multilinearity of the terms constituting $\mathfrak M$ and $\mathfrak N$.

Thus, we have reduced to proving
\begin{align}
  \label{NonlinearA}
  \norm{\mathfrak M_\beta(A,\partial_t A,F,\partial_t F)}_{H^{s-1,b'-1}}
  &\lesssim N (1+N^3),
  \\
  \label{NonlinearB}
  \norm{\mathfrak N_\beta(A,\partial_t A,F,\partial_t F)}_{H^{-1,b'-1}}
  &\lesssim N (1+N^3),
\end{align}
where
$$
  N = \norm{A_+}_{X^{s,b}_+} + \norm{A_-}_{X^{s,b}_-}
  + \norm{F_+}_{X^{0,b}_+} + \norm{F_-}_{X^{0,b}_-},
$$
and
\begin{equation}\label{Exponents}
  b = \frac12 + \varepsilon, \qquad b' = \frac12 + 2\varepsilon, \qquad 1-\varepsilon < s < 1,
\end{equation}
for $\varepsilon > 0$ sufficiently small; in fact, $\varepsilon \le 1/16$ suffices, as we will see in the course of the proof.

In the left-hand sides of the above estimates, as well as the ones that follow, it is understood that the substitution \eqref{Substitution} is used on the arguments of $\mathfrak M$ and $\mathfrak N$.

To simplify the notation, we write
\begin{align*}
  \norm{A}_{X^{s,b}} &= \norm{A_+}_{X^{s,b}_+} + \norm{A_-}_{X^{s,b}_-},
  \\
  \norm{F}_{X^{0,b}} &= \norm{F_+}_{X^{0,b}_+} + \norm{F_-}_{X^{0,b}_-}.
\end{align*}

Looking at the types of terms that appear in $\mathfrak M$ and $\mathfrak N$ (defined after \eqref{AFwaves}), and recalling also the definition of $\mathfrak Q$ from section \ref{Null} and noting that the Riesz transforms $R_i$ are bounded on all the spaces involved, we reduce to proving the following estimates, where $Q$ denotes any of the null forms $Q_0$, $Q_{0i}$ or $Q_{ij}$, and $\Pi(\dots)$ denotes a multilinear form in its arguments:
\begin{align}
  \label{NullEst1}
  \norm{Q[\abs{\nabla}^{-1} A, A]}_{H^{s-1,b'-1}}
  &\lesssim \norm{A}_{X^{s,b}} \norm{A}_{X^{s,b}},
  \\
  \label{NullEst2}
  \norm{Q[\abs{\nabla}^{-1} A, F]}_{H^{-1,b'-1}}
  &\lesssim \norm{A}_{X^{s,b}} \norm{F}_{X^{0,b}},
  \\
  \label{NullEst3}
  \norm{Q[A, A]}_{H^{-1,b'-1}}
  &\lesssim \norm{A}_{X^{s,b}} \norm{A}_{X^{s,b}},
  \\
  \label{Bilinear}
  \norm{\Pi(A,F)}_{H^{s-1,b'-1}}
  &\lesssim \norm{A}_{X^{s,b}} \norm{F}_{X^{0,b}},
  \\
  \label{Trilinear}
  \norm{\Pi(A,A,F)}_{H^{-1,b'-1}}
  &\lesssim \norm{A}_{X^{s,b}} \norm{A}_{X^{s,b}} \norm{F}_{X^{0,b}},
  \\
  \label{Quadrilinear}
  \norm{\Pi(A,A,A,A)}_{H^{-1,b'-1}}
  &\lesssim \norm{A}_{X^{s,b}} \norm{A}_{X^{s,b}} \norm{A}_{X^{s,b}} \norm{A}_{X^{s,b}}.
\end{align}

The difficult estimates are the first four, where the regularity is sharp, except in the third estimate, where there is a little bit of room. We shall in fact reduce all of them to product estimates of the form (after extracting regularity gains due to the null structure)
\begin{equation}\label{Product}
  \norm{uv}_{H^{-s_0,b_0}} \lesssim \norm{u}_{H^{s_1,b_1}} \norm{v}_{H^{s_2,b_2}}.
\end{equation}

\begin{definition} Let $s_0,s_1,s_2,b_0,b_1,b_2 \in \R$. If \eqref{Product} holds for all $u,v \in \mathcal S(\R^{1+3})$, we say that the matrix $\Hproduct{s_0 & s_1 & s_2 \\ b_0 & b_1 & b_2}$ is a \emph{product}.
\end{definition}

Estimates of this type were first studied by Klainerman and Machedon \cite{Klainerman1993}. The full range of products up to certain borderline cases was determined in \cite{Selberg2012}. For our present purposes, the following simplified version of the product law will suffice. Note, however, that it includes some borderline cases (corresponding to equality in one and only one of the last two conditions below), which will be of crucial importance for us.

\begin{theorem}\label{Atlas} \cite{Selberg2012}. Let $s_0,s_1,s_2 \in \R$ and $b_0,b_1,b_2 \ge 0$. Assume that
\begin{gather}
  \label{Atlas:1}
  \sum b_i > \frac12,
  \\
  \label{Atlas:2}
  \sum s_i > 2 - \sum b_i,
  \\
  \label{Atlas:3}
  \sum s_i > \frac32 - \min_{i \neq j}(b_i+b_j),
  \\
  \label{Atlas:4}
  \sum s_i > \frac32 - \min(b_0 + s_1 + s_2, s_0 + b_1 + s_2, s_0 + s_1 + b_2),
  \\
  \label{Atlas:5}
  \sum s_i \ge 1,
  \\
  \label{Atlas:6}
  \min_{i \neq j}(s_i+s_j) \ge 0,
\end{gather}
and that the last two inequalities are \textbf{not both equalities}. Then
$$
  P=\LHproduct{s_0 & s_1 & s_2 \\ b_0 & b_1 & b_2}
$$
is a product.

Conversely, if $P$ is a product, then the conditions \eqref{Atlas:1}--\eqref{Atlas:6}, with $>$ replaced by $\ge$ throughout, must hold.
\end{theorem}

\begin{remark}\label{AtlasRemark} If $b_0 \ge 0$ and $b_1,b_2 > 1/2$, the conditions in the theorem reduce to
\begin{gather}
  \label{AtlasRemark:1}
  \sum s_i > \frac32 - (b_0 + s_1 + s_2),
  \\
  \label{AtlasRemark:2}
  \sum s_i \ge 1,
  \\
  \label{AtlasRemark:3}
  \min_{i \neq j}(s_i+s_j) \ge 0,
\end{gather}
where the last two inequalities are not both allowed to be equalities.
\end{remark}

The key tools needed to prove \eqref{NullEst1}--\eqref{Quadrilinear} are now at hand. We start with \eqref{Bilinear}, then we prove the null form estimates, and finally the much easier trilinear and quadrilinear estimates.

\subsection{Estimate for $\Pi(A,F)$} In view of \eqref{HXH}, we can reduce \eqref{Bilinear} to
$$
  \norm{uv}_{H^{s-1,b'-1}} \lesssim \norm{u}_{H^{s,b}} \norm{u}_{H^{0,b}}.
$$
This holds by Theorem \ref{Atlas}; in fact, we can use the simplified conditions from Remark \ref{AtlasRemark}. Since
$$
  \sum s_i = 1,
$$
the estimate is sharp, and we require $\min_{i \neq j}(s_i + s_j)$ to be strictly positive, which translates to
$$
  0 < s < 1.
$$
Finally, \eqref{AtlasRemark:1} says that
$$
  s > b'-\frac12.
$$
The last two conditions are certainly satisfied with our choice \eqref{Exponents} of exponents.

\subsection{Null form estimates I: Preliminary reductions}\label{NullformReductions}

Since the matrix commutator structure plays no role in the estimates under consideration, we reduce to the ordinary null forms for $\mathbb C$-valued functions $u$ and $v$,
$$
\left\{
\begin{aligned}
  Q_0(u,v) &= - \partial_t u \partial_t v + \partial^i u \partial_i v,
  \\
  Q_{0i}(u,v) &= \partial_t u \partial_i v - \partial_i u \partial_t v,
  \\
  Q_{ij}(u,v) &= \partial_i u \partial_j v - \partial_j u \partial_i v.
\end{aligned}
\right.
$$
This reduction is justified in view of \eqref{CommutatorNullforms}, which shows that the commutator null forms are linear combinations of the ordinary ones.

Substituting
$$
  u=u_+ + u_-, \quad \partial_t u = i\angles{\nabla}(u_+ - u_-),
  \quad v=v_+ + v_-, \quad \partial_t v = i\angles{\nabla}(v_+ - v_-),
$$
one obtains
\begin{align*}
  Q_0(u,v)
  &=
  \sum_{\pm,\pm'} (\pm 1)(\pm' 1)  \bigl[ \angles{D} u_{\pm} \angles{D} v_{\pm'} - (\pm D^i)u_{\pm} (\pm' D_i)v_{\pm'} \bigr],
  \\
  Q_{0i}(u,v)
  &=
  \sum_{\pm,\pm'} (\pm 1)(\pm' 1) \bigl[ - \angles{D} u_{\pm} (\pm' D_i)v_{\pm'} + (\pm D_i)u_{\pm} \angles{D}v_{\pm'} \bigr],
  \\
  Q_{ij}(u,v)
  &=
  \sum_{\pm,\pm'} (\pm 1)(\pm' 1) \bigl[ - (\pm D_i)u_{\pm} (\pm' D_j)v_{\pm'} + (\pm D_j)u_{\pm} (\pm' D_i)v_{\pm'} \bigr],
\end{align*}
where
$$
  D = (D_1,D_2,D_3) = \frac{\nabla}{i}
$$
has Fourier symbol $\xi$. In terms of the Fourier symbols
\begin{align*}
  q_0(\xi,\eta) &= \angles{\xi}\angles{\eta} - \xi \cdot \eta,
  \\
  q_{0i}(\xi,\eta) &= - \angles{\xi} \eta_i + \xi_i \angles{\eta},
  \\
  q_{ij}(\xi,\eta) &= - \xi_i \eta_j + \xi_j \eta_i,
\end{align*}
we have, more conveniently,
\begin{align*}
  Q_0(u,v) &= \sum_{\pm,\pm'} (\pm 1)(\pm' 1) B_{q_0(\pm \xi, \pm' \eta)}( u_{\pm}, v_{\pm'} ),
  \\
  Q_{0i}(u,v) &= \sum_{\pm,\pm'} (\pm 1)(\pm' 1) B_{q_{0i}(\pm \xi, \pm' \eta)}( u_{\pm}, v_{\pm'} ),
  \\
  Q_{ij}(u,v) &= \sum_{\pm,\pm'} (\pm 1)(\pm' 1) B_{q_{ij}(\pm \xi, \pm' \eta)}( u_{\pm},  v_{\pm'}),
\end{align*}
where for a given symbol $\sigma(\xi,\eta)$ we denote by $B_{\sigma(\xi,\eta)}(\cdot,\cdot)$ the operator defined by
$$
  \mathcal F_{t,x} \left\{ B_{\sigma(\xi,\eta)}(u,v) \right\}(\tau,\xi) = \int \sigma(\xi-\eta,\eta) \widetilde u(\tau-\lambda,\xi-\eta) \widetilde v(\lambda,\eta) \, d\lambda \, d\eta.
$$

We now estimate the symbols appearing above.

\begin{lemma} For all nonzero $\xi,\eta \in \R^3$,
\begin{align*}
  \abs{q_0(\xi,\eta)} &\lesssim \abs{\xi}\abs{\eta}\theta(\xi,\eta)^2 + \frac{1}{\min(\angles{\xi},\angles{\eta})},
  \\
  \abs{q_{0j}(\xi,\eta)} &\lesssim \abs{\xi}\abs{\eta}\theta(\xi,\eta) + \frac{\abs{\xi}}{\angles{\eta}} + \frac{\abs{\eta}}{\angles{\xi}},
  \\
  \abs{q_{ij}(\xi,\eta)} &\le \abs{\xi}\abs{\eta} \theta(\xi,\eta),
\end{align*}
where $\theta(\xi,\eta) = \arccos\left(\frac{\xi\cdot\eta}{\abs{\xi}\abs{\eta}}\right) \in [0,\pi]$ is the angle between $\xi$ and $\eta$.
\end{lemma}

\begin{proof} This is easily seen by writing the symbols in the following way:
\begin{align*}
  q_0(\xi,\eta) &= \abs{\xi}\abs{\eta}\left(1 - \frac{\xi\cdot\eta}{\abs{\xi}\abs{\eta}} \right) + \angles{\xi}\angles{\eta} - \abs{\xi}\abs{\eta},
  \\
  q_{0i}(\xi,\eta) &= \abs{\xi}\abs{\eta} \left( \frac{\xi_i}{\abs{\xi}} - \frac{\eta_i}{\abs{\eta}} \right)
  + \xi_i \left( \angles{\eta} - \abs{\eta} \right) - \left( \angles{\xi} - \abs{\xi} \right) \eta_i,
  \\
  \abs{q_{ij}(\xi,\eta)} &\le \abs{ \xi \times \eta}.
\end{align*}
\end{proof}

In view of this, and since the norms we use only depend on the absolute value of the space-time Fourier transform, we can reduce any estimate for $Q(u,v)$ to a corresponding estimate for the three expressions
$$
  B_{\theta(\pm\xi,\pm'\eta)}(\abs{\nabla} u, \abs{\nabla} v),
  \qquad
  \angles{\nabla} u \angles{\nabla}^{-1} v
  \qquad \text{and} \qquad
  \angles{\nabla}^{-1} u \angles{\nabla} v.
$$
Thus, \eqref{NullEst1}--\eqref{NullEst3} can be reduced to the following:
\begin{align}
  \label{NullEst1:1}
  \norm{B_{\theta(\pm\xi,\pm'\eta)}(u, v)}_{H^{s-1,b'-1}}
  &\lesssim \norm{u}_{X^{s,b}_\pm} \norm{v}_{X^{s-1,b}_{\pm'}},
  \\
  \label{NullEst2:1}
  \norm{B_{\theta(\pm\xi,\pm'\eta)}(u, v)}_{H^{-1,b'-1}}
  &\lesssim \norm{u}_{X^{s,b}_\pm} \norm{v}_{X^{-1,b}_{\pm'}},
  \\
  \label{NullEst3:1}
  \norm{B_{\theta(\pm\xi,\pm'\eta)}(u, v)}_{H^{-1,b'-1}}
  &\lesssim \norm{u}_{X^{s-1,b}_\pm} \norm{v}_{X^{s-1,b}_{\pm'}},
  \\
  \label{NullEst1:2}
  \norm{uv}_{H^{s-1,0}}
  &\lesssim \norm{\abs{\nabla}u}_{H^{s-1,b}} \norm{v}_{H^{s+1,b}},
  \\
  \label{NullEst1:3}
  \norm{uv}_{H^{s-1,0}}
  &\lesssim \norm{\abs{\nabla}u}_{H^{s+1,b}} \norm{v}_{H^{s-1,b}},
  \\
  \label{NullEst2:2}
  \norm{uv}_{H^{-1,0}}
  &\lesssim \norm{\abs{\nabla}u}_{H^{s-1,b}} \norm{v}_{H^{1,b}},
  \\
  \label{NullEst2:3}
  \norm{uv}_{H^{-1,0}}
  &\lesssim \norm{\abs{\nabla}u}_{H^{s+1,b}} \norm{v}_{H^{-1,b}},
  \\
  \label{NullEst3:2}
  \norm{uv}_{H^{-1,0}}
  &\lesssim \norm{u}_{H^{s-1,b}} \norm{v}_{H^{s+1,b}},
  \\
  \label{NullEst3:3}
  \norm{uv}_{H^{-1,0}}
  &\lesssim \norm{u}_{H^{s+1,b}} \norm{v}_{H^{s-1,b}},
\end{align}
where we also used \eqref{HXH} to change the $X$-norms to $H$-norms in the right-hand sides of the last six estimates.

All these estimates will be handled using the product law, Theorem \ref{Atlas}. For those estimates that involve the null form $B_{\theta(\pm\xi,\pm'\eta)}$, however, we must first apply the following angle estimate, which quantifies the fact that in a null form, we can trade in hyperbolic regularity (decay with respect to distance from the cone in Fourier space) and gain a corresponding amount of elliptic regularity.

\begin{lemma}\label{AngleEstimate} Let $\alpha,\beta,\gamma \in [0,1/2]$. Then for all pairs of signs $(\pm,\pm')$, all $\tau,\lambda \in \R$ and all nonzero $\xi,\eta \in \R^3$,
$$
  \theta(\pm\xi,\pm'\eta)
  \lesssim
  \left( \frac{\angles{\abs{\tau+\lambda}-\abs{\xi+\eta}}}{\min(\angles{\xi},\angles{\eta}} \right)^\alpha
  +
  \left( \frac{\angles{-\tau\pm\abs{\xi}}}{\min(\angles{\xi},\angles{\eta})} \right)^\beta
  +
  \left( \frac{\angles{-\lambda\pm'\abs{\eta}}}{\min(\angles{\xi},\angles{\eta})} \right)^\gamma.
$$
\end{lemma}

For a proof, see for example \cite[Lemma 2.1]{Selberg2008}.

Combining this angle estimate with the product law, we deduce the following.

\begin{theorem}\label{NullformThm}
Let $\sigma_0,\sigma_1,\sigma_2,\beta_0,\beta_1,\beta_2 \in \R$. Assume that
\begin{gather}
 \label{NullformThm:1}
  0 \le \beta_0 < \frac12 < \beta_1,\beta_2 < 1,
  \\
  \label{NullformThm:2}
  \sum \sigma_i + \beta_0 > \frac32 - (\beta_0 + \sigma_1 + \sigma_2),
  \\
  \label{NullformThm:3}
  \sum \sigma_i > \frac32 - (\sigma_0 + \beta_1 + \sigma_2),
  \\
  \label{NullformThm:4}
  \sum \sigma_i > \frac32 - (\sigma_0 + \sigma_1 + \beta_2),
  \\
  \label{NullformThm:5}
  \sum \sigma_i + \beta_0 \ge 1,
  \\
  \label{NullformThm:6}
  \min(\sigma_0 + \sigma_1, \sigma_0 + \sigma_2, \beta_0 + \sigma_1 + \sigma_2) \ge 0,
\end{gather}
and that the last two inequalities are \textbf{not both equalities}. Then we have the null form estimate
$$
  \norm{B_{\theta(\pm\xi,\pm'\eta)}(u,v)}_{H^{-\sigma_0,-\beta_0}}
  \lesssim
  \norm{u}_{X^{\sigma_1,\beta_1}_\pm} \norm{v}_{X^{\sigma_2,\beta_2}_{\pm'}}.
$$
\end{theorem}

\begin{proof}
Applying Lemma \ref{AngleEstimate}, and using also \eqref{HXH} and the symmetric nature of the conditions on $(\sigma_1,\beta_1)$ and $(\sigma_2,\beta_2)$, we reduce to checking that the following are products:
\begin{align*}
  P_1 &= \LHproduct{\sigma_0 & \beta_0 + \sigma_1  & \sigma_2 \\ 0 & \beta_1 & \beta_2},
  \\
  P_2 &= \LHproduct{\sigma_0 & \sigma_1+1/2 & \sigma_2 \\ \beta_0 & \beta_1-1/2 & \beta_2},
  \\
  P_3 &= \LHproduct{\sigma_0 & \sigma_1 & \sigma_2+1/2 \\ \beta_0 & \beta_1-1/2 & \beta_2}.
\end{align*}
But this follows from Theorem \ref{Atlas}, as one can readily check. Note that for $P_1$ the simplification in Remark \ref{AtlasRemark} applies.
\end{proof}

\subsection{Null form estimates II: Conclusion}

Using Theorem \ref{NullformThm} it is now a routine matter to check that the estimates \eqref{NullEst1:1}, \eqref{NullEst2:1} and \eqref{NullEst3:1} hold. In fact, the estimates in question correspond to the following matrices $\Hproduct{\sigma_0 & \sigma_1  & \sigma_2 \\ \beta_0 & \beta_1 & \beta_2}$:
\begin{align*}
  N_1 &= \LHproduct{1-s & s & s-1 \\ 1-b' & b & b},
  \\
  N_2 &= \LHproduct{1 & s & -1 \\ 1-b' & b & b},
  \\
  N_3 &= \LHproduct{1 & s-1 & s-1 \\ 1-b' & b & b},
\end{align*}
and checking against the conditions in Theorem \ref{NullformThm} gives the following conditions:
$$
  s > \max \left( \frac32 - b, b', \frac14 + b', \frac58 + \frac{b'}{2}, \frac56 - \frac{b}{3} \right),
$$
which is consistent with our assumption \eqref{Exponents}, provided $0 < \varepsilon \le 1/16$.

We remark that for $N_1$ and $N_2$, equality holds in \eqref{NullformThm:6}, so these estimates are sharp.

It remains to check the estimates \eqref{NullEst1:2}--\eqref{NullEst3:3}. Ignoring low-frequency issues for the moment, so that we can replace $\abs{\nabla} u$ by $\angles{\nabla} u$, we reduce to
\begin{align*}
  \norm{uv}_{H^{s-1,0}}
  &\lesssim \norm{u}_{H^{s,b}} \norm{v}_{H^{s+1,b}},
  \\
  \norm{uv}_{H^{s-1,0}}
  &\lesssim \norm{u}_{H^{s+2,b}} \norm{v}_{H^{s-1,b}},
  \\
  \norm{uv}_{H^{-1,0}}
  &\lesssim \norm{u}_{H^{s,b}} \norm{v}_{H^{1,b}},
  \\
  \norm{uv}_{H^{-1,0}}
  &\lesssim \norm{u}_{H^{s+2,b}} \norm{v}_{H^{-1,b}},
  \\
  \norm{uv}_{H^{-1,0}}
  &\lesssim \norm{u}_{H^{s-1,b}} \norm{v}_{H^{s+1,b}},
  \\
  \norm{uv}_{H^{-1,0}}
  &\lesssim \norm{u}_{H^{s+1,b}} \norm{v}_{H^{s-1,b}}.
\end{align*}
These estimates hold by Theorem \ref{Atlas}. In fact, the reduction in Remark \ref{AtlasRemark} applies. There is a lot of room in all the conditions except $\min_{i \neq j} (s_i+s_j) \ge 0$, which holds with equality for the second and fourth estimates.

It remains to prove \eqref{NullEst1:2}--\eqref{NullEst2:3} the case where $u$ is at low frequency, that is, $\abs{\xi} \le 1$ on the Fourier support of $u$. But then the frequency $\eta$ of $v$ is comparable to the output frequency $\xi+\eta$, in the sense that $\angles{\xi+\eta} \sim \angles{\eta}$, hence we reduce to
$$
  \norm{uv}_{L^2} \lesssim \norm{\abs{\nabla}u}_{L^2} \norm{v}_{H^{0,b}},
$$
but this follows by estimating
$$
  \norm{uv}_{L^2} \le
  \norm{u}_{L_t^2(L_x^\infty)} \norm{v}_{L_t^\infty(L_x^2)}
$$
followed by Sobolev embedding and the estimate
\begin{equation}\label{Hembedding}
  \norm{v}_{L_t^q(L_x^2)} \lesssim \bignorm{ \norm{\widetilde v(\tau,\xi)}_{L_\tau^{q'}} }_{L^2_\xi} \le \norm{\angles{\tau}^{-b}}_{L_\tau^{2 + 4/(q-2)}} \norm{v}_{H^{0,b}},
\end{equation}
valid for $2 \le q \le \infty$.

\subsection{Estimate for $\Pi(A,A,F)$} Here we prove \eqref{Trilinear}:
$$
  \norm{uvw}_{H^{-1,b'-1}} \lesssim \norm{u}_{H^{s,b}} \norm{v}_{H^{s,b}} \norm{w}_{H^{0,b}}.
$$
This holds for $s \ge 1/2+2\varepsilon$ by the estimates
\begin{align*}
  \norm{uw}_{H^{-1,b'-1}} &\lesssim \norm{u}_{H^{3\varepsilon,0}} \norm{w}_{H^{0,b}},
  \\
  \norm{uv}_{H^{3\varepsilon,0}} &\lesssim \bignorm{u}_{H^{1/2+2\varepsilon,b}} \bignorm{v}_{H^{1/2+2\varepsilon,b}},
\end{align*}
both of which hold by Theorem \ref{Atlas}.

\subsection{Estimate for $\Pi(A,A,A,A)$} We prove \eqref{Quadrilinear}:
\begin{align*}
  \norm{u_1u_2u_3u_4}_{H^{-1,0}}
  &\lesssim
  \norm{u_1u_2u_3u_4}_{L_t^2(L_x^{6/5})}
  \\
  &\le
  \norm{u_1}_{L_t^8(L_x^{24/5})} \norm{u_2}_{L_t^8(L_x^{24/5})} \norm{u_3}_{L_t^8(L_x^{24/5})} \norm{u_4}_{L_t^8(L_x^{24/5})}
  \\
  &\lesssim
  \norm{u_1}_{H^{7/8,b}} \norm{u_2}_{H^{7/8,b}} \norm{u_3}_{H^{7/8,b}} \norm{u_4}_{H^{7/8,b}}
\end{align*}
where we applied Sobolev embedding and \eqref{Hembedding}.

\bibliographystyle{amsplain} 
\bibliography{YMbibliography}

\end{document}